\documentclass[12pt]{article}
\usepackage[utf8]{inputenc}
\usepackage{geometry}
\usepackage{amssymb}
\usepackage{amsthm}
\usepackage{amsmath}
\usepackage{enumerate}
\usepackage{booktabs}
\usepackage{graphicx}
\usepackage{multirow}
\usepackage{color}
\usepackage{authblk}

\newtheorem{theorem}{Theorem}

\newtheorem{definition}{Definition}

\newtheorem{example}{Example}

\newcommand{\R}{\mathbb{R}}

\title{Deep Hamiltonian networks based on symplectic integrators}

\author[1,2]{Aiqing Zhu}
\author[1,2]{Pengzhan Jin}
\author[1,2,*]{Yifa Tang}
\affil[1]{LSEC, ICMSEC, Academy of Mathematics and Systems Science, Chinese Academy of Sciences, Beijing 100190, China}
\affil[2]{School of Mathematical Sciences, University of Chinese Academy of Sciences, Beijing 100049, China}
\affil[*]{Email address: tyf@lsec.cc.ac.cn}
\date{}

\begin{document}

\maketitle

\begin{abstract}

HNets is a class of neural networks on grounds of physical prior for learning Hamiltonian systems. This paper explains the influences of different integrators as hyper-parameters on the HNets through error analysis. If we define the network target as the map with zero empirical loss on arbitrary training data, then the non-symplectic integrators cannot guarantee the existence of the network targets of HNets. We introduce the inverse modified equations for HNets and prove that the HNets based on symplectic integrators possess network targets and the differences
between the network targets and the original Hamiltonians depend on the accuracy orders of the integrators. Our numerical experiments show that the phase flows of the Hamiltonian systems obtained by symplectic HNets do not exactly preserve the original Hamiltonians, but preserve the network targets calculated; the loss of the network target for the training data and the test data is much less than the loss of the original Hamiltonian; the symplectic HNets have more powerful generalization ability and higher accuracy than the non-symplectic HNets in addressing predicting issues. Thus, the symplectic integrators are of critical importance for HNets.
\end{abstract}


\section{Introduction}
Dynamical systems play a critical role in shaping our understanding of the physical world. And recent line of works bridged the connection between dynamical systems and deep neural networks. It is widely studied to analyze neural networks from the perspective of dynamic systems \cite{weinan2017proposal, pascanu2013difficulty, ruthotto2019deep}. And researchers make an effort to employ deep learning to dynamical systems recently \cite{lutter2019deep, raissi2018hidden,raissi2017inferring}. In particular, neural networks have been applied to solve differential equations \cite{weinan2018deep,pang2019fpinns,raissi2019physics,zhang2019quantifying}. With the explosive growth of available data and computing resources, current papers focus on discovery sufficiently accurate models of dynamical systems directly from data.

A good physics model could predict changes in a system over time. In particular, our goal is discovery of dynamics systems on grounds of remarkable generalization ability of neural networks. In this task, multistep neural networks introduce a novel approach to nonlinear systems identification that combines the classical multi-step methods with deep neural networks \cite{raissi2018multistep}. ODENets based on general ODE solver, in contrast, propose using adjoint equation instead of back-propagating through ODE solver \cite{chen2018neural}.

The problem with extant methods is that they tend not to learn conservation laws. This often causes them to drift away from the true dynamics of the system as errors accumulate \cite{greydanus2019hamiltonian}. Hamiltonian system is one of the expressions of classical mechanics and has been applied to a wide range of physics fields from celestial mechanics to quantum field theory \cite{arnold2007mathematical, reichl1999modern, sakurai1995modern}, and there are also important applications for machine learning \cite{bertalan2019learning, li2019neural,rezende2019equivariant,sanchez2019hamiltonian, toth2019hamiltonian}. Hamiltonian system is in the form
\begin{equation}\label{eq:Hami}
\dot{y} = J^{-1} \nabla H(y), \quad J=\begin{pmatrix} 0 & I_d \\ -I_d & 0\end{pmatrix},
\end{equation}
where $y \in  \R^{2d}$, $I_d \in \R^{d \times d}$ is the $d$-by-$d$ identity matrix. The scalar function $H(y)$ is called the Hamiltonian \cite{arnol2013mathematical}. In order to learn Hamiltonian systems, \cite{greydanus2019hamiltonian} proposes the HNet to learn a parametric function for $H(y)$. \cite{chen2019symplectic} improves HNet for separable Hamiltonian as SRNN, and it numerically confirms that HNets based on symplectic integrators perform better than the ones based on non-symplectic integrators.

For the numerical solution of the Hamiltonian system, symplectic integrator has a unique and irreplaceable advantage, especially for the long-term tracking of the system and the conservation of invariant. Pioneering work on symplectic integration is due to Kang Feng \cite{feng1984difference}, and this direction has been extensively studied and has achieved extremely fruitful results \cite{feng1986difference,feng1995collected,feng2010symplectic,hairer2006geometric,koch2007dynamical,lubich2008quantum,sanz2018numerical}. Symplectic integators solve the long-term calculation of dynamic systems, and have also been successfully applied in diverse fields of science and engineering \cite{faou2009computing,omelyan2003symplectic,qin2015canonical,zhang2014canonicalization,zhu2016splitting}.

Following are the definitions of symplectic map and symplectic integrator.
\begin{definition}
A differentiable map $g : U \rightarrow \R^{2d}$ (where $U\subseteq \R^{2d}$ is an open set) is called symplectic if
\begin{equation*}
g'(y)^{T}Jg'(y)=J,
\end{equation*}
where $g'(y)$ is the Jacobian of $g(y)$.
\end{definition}
In 1899, Poincare proved that the flow of the Hamiltonian system is a symplectic map \cite[Chapter \uppercase\expandafter{\romannumeral6}.2]{hairer2006geometric}, i.e.,
\begin{equation*} \label{eq:sym_con}
\left(\frac{\partial \phi_t}{\partial y_0}\right)^T J \left(\frac{\partial \phi_t}{\partial y_0}\right) = J,
\end{equation*}
where $\phi_{t}(y_0)$ is the flow of (\ref{eq:Hami}) starting from$y_0$ by time $t$.
\begin{definition}
An integrator $y_1 = \Phi_h(y_0)$ is called symplectic if the one-step map $\Phi_h(y)$ is symplectic whenever the integrator is applied to a smooth Hamiltonian system.
\end{definition}
In this paper, symplectic Euler method and implicit  midpoint rule are both symplectic, while explicit Euler method and implicit trapezoidal rule are both non-symplectic. More information about the symplectic integrator refers to \cite{feng2010symplectic}.

The recent study has verified the importance of symplectic integrators in HNets by numerical experiments \cite{chen2019symplectic}, but the theoretical understanding is still lagging behind. The core of this work is to build the backward error analysis of HNets. We introduce the target error and the network target. The target error is the difference between the network target and the true target, where the network target is defined as the map with zero empirical loss on arbitrary training data. In addition, the \textbf{inverse modified equation} is proposed to calculate the network target. It is proved that the HNets based on the symplectic integrators possess network targets while the non-symplectic integrators cannot guarantee the existence of the network targets. We also perform the experiments to confirm the theoretical results later.

The paper is organized as follows. Section \ref{sec:nt_ime} introduces the concepts of the target error and the network target, furthermore, proposes the inverse modified equations for backward analysis. Section \ref{sec:num_res} presents the numerical results of the target error and the prediction of the phase flows of the Hamiltonian systems on grounds of HNets. Some conclusions will be given in the last section.

\section{Network targets and inverse modified equations}
\label{sec:nt_ime}
\subsection{Target error}
The neural networks, as universal approximators \cite{barron1993universal,cybenko1989approximation,hornik1989multilayer}, can approximate essentially any function. First we show the definition of the target error.
\begin{definition}
The network target (NT) is the map with zero empirical loss on arbitrary training data. The true target (TT) is the map expected to be approached. The difference between them is called the target error (TE).
\end{definition}
To approximate the function $f(y)$, the loss function is generally defined as
\begin{equation}\label{loss:gen}
\frac{1}{|\mathcal{T}|}\sum_{(y_i,f(y_i)) \in \mathcal{T}} \|net(y_i)-f(y_i)\|,
\end{equation}
where $\mathcal{T} = \{(y_i,f(y_i)\}_{i=1}^{N}$ is the training dataset. It is clear that in this case the network target and the true target are both $f(y)$, i.e., the network $net(y)$ is an approximation of $f$ but not other functions. However, the network target is not the same as the true target for some networks with priors, and there is even no network target.

Multistep neural network (MNN) \cite{raissi2018multistep}, and Hamiltonian neural network (HNet) \cite{greydanus2019hamiltonian}, are two examples of non-zero target errors. In consideration of the ordinary differential equation
\begin{equation}\label{eq:ODE}
\dot{y}=f(y),
\end{equation}
where $y \in \R^{n}$. MNN, whose true target is $f(y)$, proceeds by applying a linear multistep method to
\begin{equation*}\label{loss:ODENets}
\|\frac{dy}{dt} - net(y)\|
\end{equation*}
and obtain the loss function. For instance, the loss function of MNN based on explicit Euler method $\Phi_h(f,y)= y + h f(y)$ is
\begin{equation*}
\frac{1}{|\mathcal{T}|}\sum_{(y_n,\phi_h(y_n)) \in \mathcal{T} }\| \frac{\phi_h(y_n)-y_n}{h} - net(y_n)\|,
\end{equation*}
where $\phi_h(y)$ is the exact flow of equation (\ref{eq:ODE}) and $\mathcal{T} = \{(y_i,\phi_h(y_i)\}_{i=1}^{N}$ is the training data. And the network target satisfies
\begin{equation*}
NT(y)=\frac{\phi_h(y)-y}{h}=f(y)+\frac{h}{2}f'(y)f(y)+\frac{h^2}{6}(f''(y)(f(y),f(y))+f'(y)f'(y)f(y))+ \cdots ,
\end{equation*}
where $f(y)$ is a vector-valued function whose higher-order derivatives are tensors. Thus the target error of MNN based on explicit Euler method can be expressed as
\begin{equation*}
NT(y)-f(y) = \frac{h}{2}f'(y)f(y)+\frac{h^2}{6}(f''(y)(f(y),f(y))+f'(y)f'(y)f(y))+ \cdots.
\end{equation*}
\begin{figure}[htbp]
    \centering
    \includegraphics[width=0.98\textwidth]{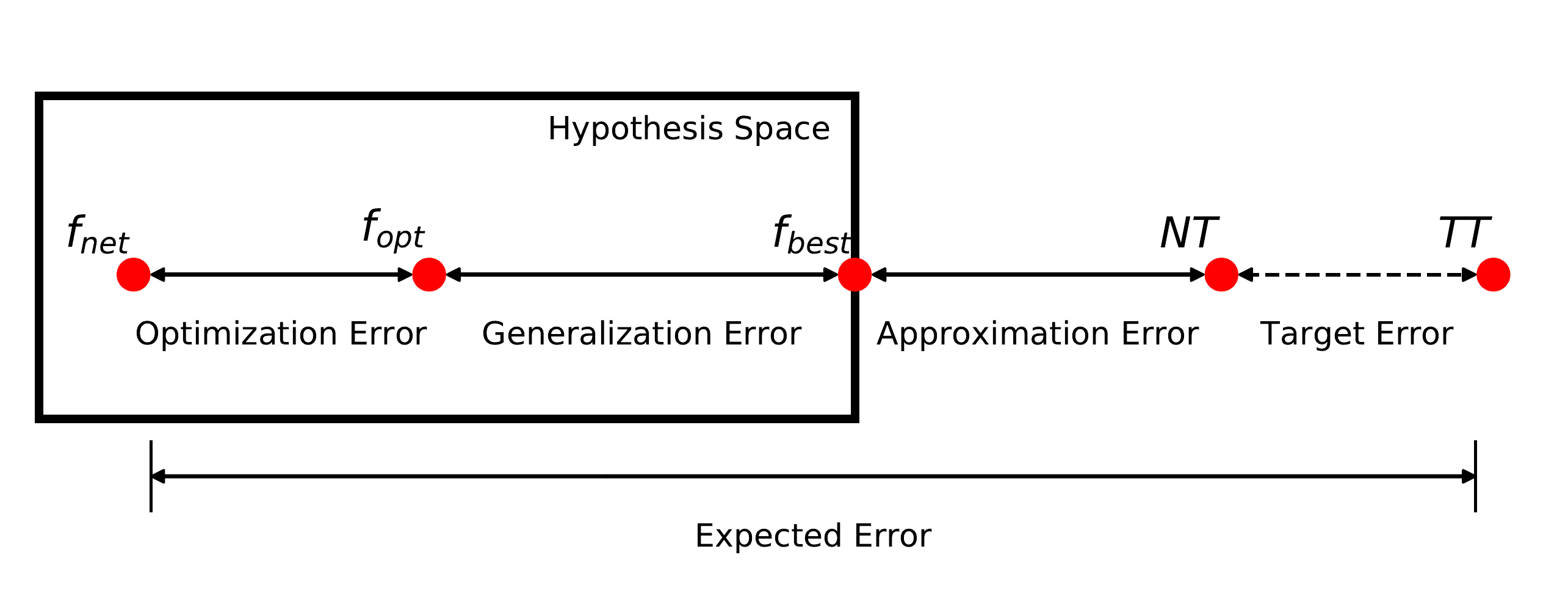}
    \caption{\textbf{Illustration of expected error. }$f_{net}$ is the function by training a neural network, $f_{opt}$ is the neural network whose loss is at a global minimum, $f_{best}$ is the function closest to $NT$ in the hypothesis space, $NT$ is the network target and $TT$ is the true target. The expected error consists of four parts, of which optimization error, generalization error and approximation error are the main objects of classic neural network error analysis, and the final target error is usually zero so that it is often ignored. When $NT$ is different from $TT$, sufficient training, a great quantity of data and large network size can effectively reduce the errors of the first three, consequently the target error will become the main part of the expected error.
    }
    \label{fig:error_types}
\end{figure}

The expected error mainly depends on optimization error, generalization error and approximation error, while the target error is usually zero so that it is often ignored, as shown in Fig.~\ref{fig:error_types}. There have been numerous studies that analyze the optimization, generalization and approximation errors \cite{bottou2010large,bottou2008tradeoffs,cybenko1989approximation,hornik1989multilayer,jin2019quantifying,lee2016gradient,poggio2017theory}, but the target error is lagging behind. When neural networks are used to learn dynamic systems, sufficient data, developed optimization techniques as well as powerful approximation capabilities, make the target error a major part of expected error. That is what we should focus on.

The non-symplectic integrators cannot guarantee the existence of the network targets of HNets. For instance, if the chosen numerical integrator is explicit Euler method, the loss function is
\begin{equation*}
\frac{1}{|\mathcal{T}|}\sum_{(y_n,\phi_h(y_n)) \in \mathcal{T} }\| \frac{\phi_h(y_n)-y_n}{h} -J^{-1} \nabla net(y_n)\|
\end{equation*}
with exact flow $\phi_h(y)$ and training data $\mathcal{T}$, then the network target is subject to
\begin{equation*}
\nabla NT(y) = J \frac{\phi_h(y)-y}{h}.
\end{equation*}
However, not every vector-valued function is the gradient of another scalar function, that means the network target $NT$ may not exist. As shown in Fig.~\ref{fig:error_types}, the absence of network targets makes classic error analysis no longer applicable. This work will prove the existence of network targets of HNets based on the symplectic integrators.

\subsection{Inverse modified equation}
Consider an ordinary differential equation
\begin{equation}\label{eq:ODE2}
\dot{y} = f(y)
\end{equation}
and a numerical integrator $\Phi_h(f, y)$ which produces the numerical approximations as $$y_0=y(0),\quad y_{i+1}=\Phi_h(f, y_{i}).$$
The idea of modified differential equation is to search for a equation of the form
\begin{equation*}\label{eq:Modieq}
\dot{\bar{y}}=\bar{f_h}(\bar{y}),
\end{equation*}
such that $\bar{y}(nh)=y_n$. In contrast, now we are aiming to search for an inverse modified differential equation of the form
\begin{equation}\label{eq:inmodieq}
f_h(\tilde{y})=f_1(\tilde{y})+hf_2(\tilde{y})+h^2f_3(\tilde{y})+\cdots
\end{equation}
such that $\tilde{y}_n = y(nh)$ for $\tilde{y}_{i+1}=\Phi_h(f_h, \tilde{y}_{i})$ and the exact solution $y(t)$ of (\ref{eq:ODE2}). Consequently, the inverse modified differential equation is indeed the network target of the multi-step network.

For the computation of (\ref{eq:inmodieq}), we expand the solution of (\ref{eq:ODE2}) into a Taylor series with respect to time step $h$:
\begin{equation}\label{eq:exasolu}
\phi_h(f,y)=y+hf(y)+\frac{h^2}{2}f'f(y)+\frac{h^3}{6}(f''(f,f)(y)+f'f'f(y))+\cdots .
\end{equation}
Moreover, assume that the numerical integrator $\Phi_h(f_h, y)$ can be expanded as
\begin{equation}\label{eq:numsolu}
\Phi_h(f_h, y) = y + hd_1(f_h,y) + h^2d_2(f_h,y) + h^3d_3(f_h,y) + \cdots,
\end{equation}
where the functions $d_j$ are given and typically composed of $f_h$ and its derivatives. In order to achieve $\tilde{y}_n = y(nh)$, it should be satisfied that $\Phi_h(f_h, y)=\phi_h(f,y)$. Now plugging (\ref{eq:inmodieq}) into (\ref{eq:numsolu}), and we can easily obtain the expressions of $f_i$ in (\ref{eq:inmodieq}) by comparing like powers of $h$ in (\ref{eq:exasolu}) and (\ref{eq:numsolu}).
\begin{example}
The implicit midpoint rule
\begin{equation*}
\Phi_h(f_h, y)=y+hf_h(\frac{\Phi_h(f_h, y)+y}{2}),
\end{equation*}
could be expanded as
\begin{equation*}
\begin{aligned}
\Phi_h(f_h, y)=&y+hf_h(y)+ hf_h'(y)\frac{\Phi_h(f_h, y)-y}{2} + \frac{h}{2}f_h''(y)(\frac{\Phi_h(f_h, y)-y}{2},\frac{\Phi_h(f_h, y)-y}{2}) + \cdots\\
 =& y+hf_h(y) + h^2\frac{f_h'f_h(y)}{2} + h^3(\frac{f_h''(f_h,f_h)(y)}{8} + \frac{f_h'f_h'f_h(y)}{4})+\cdots\\
 =& y + hf_1+h^2(f_2(y)+ \frac{1}{2}f_1'f_1)\\
&+h^3(f_3(y)+\frac{1}{2}f_2'f_1(y)+\frac{1}{2}f_1'f_2(y)+\frac{1}{4}f_1'f_1'f_1(y)+\frac{1}{8}f_1''(f_1,f_1)(y))+\cdots .
\end{aligned}
\end{equation*}
Comparing like powers of $h$ in the expression (\ref{eq:exasolu}) and the above yields recurrence relations for functions $f_j$, namely,
\begin{equation*}
\begin{aligned}
f_1(y) &= f(y)\\
f_2(y) &= \frac{1}{2}f'f(y)-\frac{1}{2}f_1'f_1(y) = 0\\
f_3(y) &= \frac{1}{6}(f''(f,f)(y)+f'f'f(y)) - (\frac{1}{2}f_2'f_1(y)+\frac{1}{2}f_1'f_2(y)+\frac{1}{4}f_1'f_1'f_1(y)+\frac{1}{8}f_1''(f_1,f_1)(y))\\
& = -\frac{1}{12}f'f'f(y)+\frac{1}{24}f''(f,f)(y))\\
& \vdots
\end{aligned}
\end{equation*}
\end{example}
We only do formal analysis without taking care of convergence issues in this work.
\begin{theorem}\label{the:modiode}
Suppose that the integrator $\Phi_h(f,y)$ is of order $p$, more precisely,
\begin{equation*}
\Phi_h(f,y)=\phi_h(f,y)+h^{p+1}\delta_{p+1}(f,y)+O(h^{p+2}),
\end{equation*}
where $\phi_h(f,y)$ denotes the exact flow of $\dot{y}=f(y)$, and $h^{p+1}\delta_{p+1}(f,y)$ is the leading term of the local truncation. The inverse modified equation satisfies
\begin{equation*}
\dot{\tilde{y}}=f_h(\tilde{y})=f(\tilde{y})+h^pf_{p+1}(\tilde{y})+\cdots,
\end{equation*}
where $f_{p+1}(y)=-\delta_{p+1}(f,y)$.
\end{theorem}
\begin{proof}
\begin{equation*}
\begin{aligned}
\phi_h(f,y) =& \Phi_h(f_h,y)\\
=& \phi_h(f_h,y) + h^{p+1} \delta_{p+1}(f_h,y)+O(h^{p+2})\\
=&  h^{p+1} \delta_{p+1}(f_h,y) +O(h^{p+2})+ y+hf_h(y)+ \\
&\frac{h^2}{2}f_h'(y)f_h(y)+\frac{h^3}{6}(f_h''(f_h,f_h)(y)+f_h'f_h'f_h(y))+\cdots.
\end{aligned}
\end{equation*}
Inserting (\ref{eq:inmodieq}) and (\ref{eq:exasolu}) into it and comparing the coefficient of the first power of $h$ yields $f_1 = f$. Thus $\delta_{p+1}(f_h,y) = \delta_{p+1}(f,y) + O(h)$. Furthermore, comparing like powers of $h$ yields $f_2=f_3= \cdots =f_p=0$ and $f_{p+1}=-\delta_{p+1}(f,y)$.
\end{proof}
The above theorem shows that the high-order integrator can effectively reduce the target error. The network target of HNet is the Hamiltonian of the inverse modified equation, nevertheless, the non-symplectic integrators cannot guarantee the inverse modified equation being a Hamiltonian system. And we point out that the inverse modified equation based on the symplectic integrator is still a Hamiltonian system.
\begin{theorem}\label{the:modiHam}
If a symplectic integrator $\Phi_h(y)$ is applied to a Hamiltonian system with a smooth Hamiltonian $H$, then the inverse modified equation (\ref{eq:inmodieq}) is also a Hamiltonian system. More precisely, there exist smooth functions $H_j$, $j=1,2,3\cdots$, such that
\begin{equation*}f_j(y)=J^{-1}\nabla H_j(y).\end{equation*}
\end{theorem}

\begin{proof}
According to Theorem \ref{the:modiode}, $f_1 =J^{-1}\nabla H(y)$. Assume that $f_j(y)=J^{-1}\nabla H_j(y)$ for $j=1,2,\cdots,r$, we need to prove the existence of $H_{r+1}(y)$ satisfying $f_{r+1}(y)=J^{-1}\nabla H_{r+1}(y)$.

Consider the truncated inverse modified equation
\begin{equation*}
\dot{\tilde{y}}=f(\tilde{y})+hf_2(\tilde{y})+h^2f_3(\tilde{y})+\cdots+h^{r-1}f_r(\tilde{y}),
\end{equation*}
which has the Hamiltonian $H(y)+hH_2(y)+\cdots+h^{r-1}H_r(y)$ by induction. Its numerical flow $\Phi_{r,h}(y)$ satisfies
\begin{equation*}\phi_h(y)= \Phi_h(f_h,y) = \Phi_{r,h}(y)+h^{r+1}f_{r+1}(y)+O(h^{r+2}).\end{equation*} And
\begin{equation*}\phi_h'(y)= \Phi_{r,h}'(y)+h^{r+1}f_{r+1}'(y)+O(h^{r+2}),\end{equation*}
where $\phi_h(y)$ and $\Phi_{r,h}(y)$ are symplectic maps, and $\Phi_{r,h}'(y)=I+O(h)$. Therefore
\begin{equation*}J=\phi_h'(y)^TJ\phi_h'(y)=J+h^{r+1}(f_{r+1}'(y)^TJ+Jf_{r+1}'(y))+O(h^{r+2}).\end{equation*}
Consequently, $f_{r+1}'(y)^TJ+Jf_{r+1}'(y)=0$, in other words, $Jf_{r+1}'(y)$ is symmetric. And the existence of $H_{r+1}(y)$ satisfying
\begin{equation*}
f_{r+1}'(y)= J^{-1}\nabla H_{r+1}(y)
\end{equation*}
follows from the Integrability Lemma \cite[Lemma \uppercase\expandafter{\romannumeral6}.2.7]{hairer2006geometric}.
\end{proof}

\section{Numerical results}
\label{sec:num_res}
\subsection{Target error}
In this subsection, we check the target error of HNet. For HNet based on symplectic integrator, let $net(y)$ be the trained network, $H(x)$ be the true target and $H_h(x)$ be the network target. Then
\begin{equation*}
net(y) - H(y) = (H_h(y)-H(y)) + (net(y) - H_h(y)) = \uppercase\expandafter{\romannumeral1} + \uppercase\expandafter{\romannumeral2},
\end{equation*}
where $\uppercase\expandafter{\romannumeral2}$ depends on the performance of the trained network, and the target error $\uppercase\expandafter{\romannumeral1}$ becomes the main factor of the expected error.

The mathematical pendulum (mass $m=1$, massless rod of length $l=1$, gravitational acceleration $g=1$) is a system having the Hamiltonian
\begin{equation*}
H(p,q)=\frac{1}{2}p^2 - \cos q,
\end{equation*}
and the differential equation is of the form
\begin{equation*}
\left\{ \begin{aligned}\dot{p}&=-\sin q \\ \dot{q}&=p\end{aligned} \right..
\end{equation*}

The training data of HNet is $\mathcal{T}=\{(y_{i},\phi_h(y_{i}))\}_{1}^{4000}$, where $y_i=(p_i,q_i)$ are randomly generated from compact set $[-\pi/2,\pi/2]\times[-\sqrt{2},\sqrt{2}]$ and $\phi_h(y)$ is the exact flow, $h=0.1$. The test data is generated in the same way. The chosen integrator is the symplectic Euler method
\begin{equation*}
\begin{aligned}
\bar{p}&= p - h \frac{\partial H(\bar{p},q)}{\partial q}\\
\bar{q}&=q + h \frac{\partial H(\bar{p},q)}{\partial p}
\end{aligned}
\end{equation*}
which is of order 1. Compute the truncations of the inverse modified equation of order 1 and 2, denoted as
\begin{equation*}
\begin{aligned}
&MH1(p,q)=\frac{1}{2}p^2 - \cos q+ \frac{h}{2}p\sin q,\\
&MH2(p,q)=\frac{1}{2}p^2 - \cos q+ \frac{h}{2}p\sin q + \frac{h^2}{6}(p^2\cos q+\sin^2 q).
\end{aligned}
\end{equation*}
The loss function of HNet is
\begin{equation*}
\frac{1}{4000}\sum_{i=1}^{4000} (\frac{\tilde{p}_i- p_i}{h} +  \frac{\partial H(\tilde{p}_i,q_i)}{\partial q})^2 + (\frac{\tilde{q}_i- q_i}{h} -  \frac{\partial H(\tilde{p}_i,q_i)}{\partial p})^2,
\end{equation*}
where $(\tilde{p}_i,\tilde{q}_i) = \phi_h(p_i,q_i)$.
\begin{table}[!htb]
    \centering
    \begin{tabular}{|c|c|c|}
    \hline
         & Training loss & Test loss  \\
     \hline
        $net(p,q)$ & $3.7\times10^{-7}$ & $3.2\times10^{-7}$   \\
     \hline
        $H(p,q)$ & $1.0\times10^{-3}$ & $1.0\times10^{-3}$ \\
     \hline
        $MH1(p,q)$ & $2.4\times10^{-6}$ & $2.1\times10^{-6}$ \\
     \hline
        $MH2(p,q)$ & $8.2\times10^{-9}$ & $9.1\times10^{-9}$ \\
     \hline
    \end{tabular}
    \caption{The training loss and test loss of HNet and three different Hamiltonians.}
    \label{tab:errMSE}
\end{table}
\begin{figure}[!htb]
    \centering
    \includegraphics[width=0.98\textwidth]{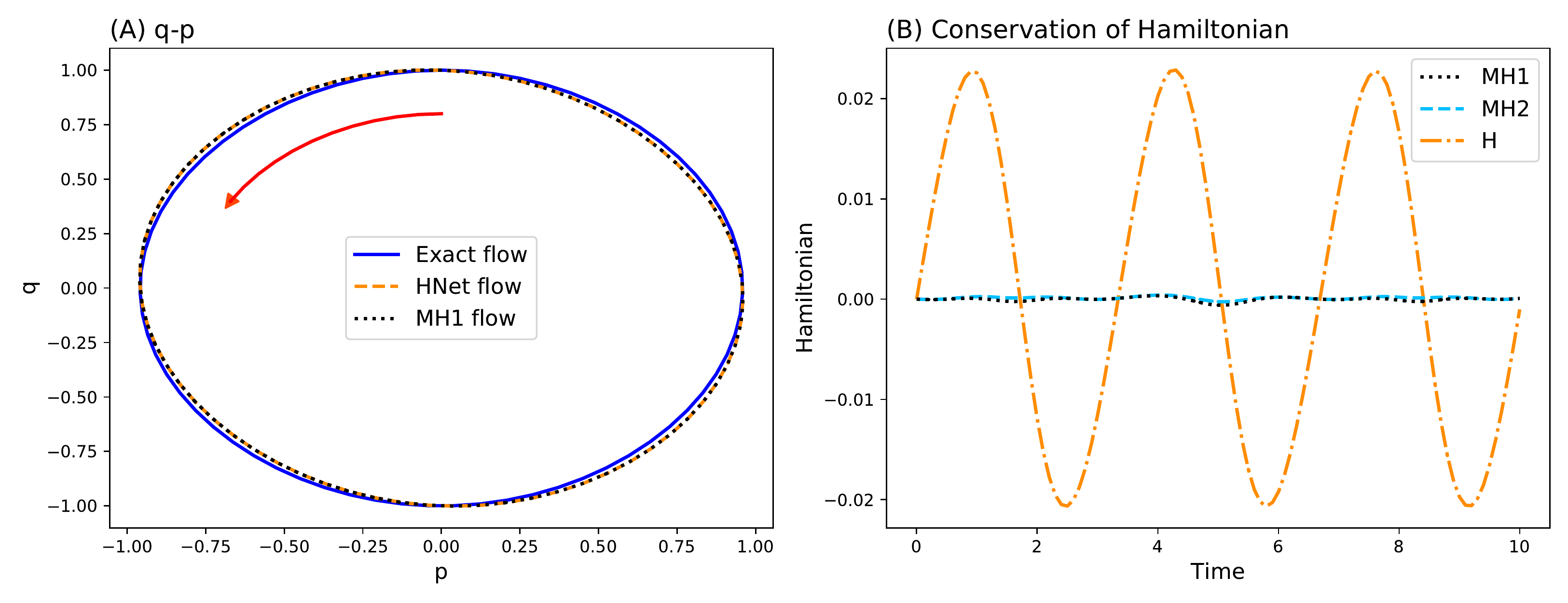}
    \caption{\textbf{Pendulum.} (\textbf{A})
    Three flows of the original pendulum system, the learned HNet and the 1-order modified system $MH1$ respectively. The HNet correctly captures the flow of the modified system rather than the original pendulum system. (\textbf{B}) Conservation of the original Hamiltonian of pendulum compares to the two corresponding truncated Hamiltonians of the inverse modified equation ($MH1$ and $MH2$). The HNet nearly conserves the Hamiltonian of the modified system rather than the original pendulum system.}
    \label{fig:SEPD}
\end{figure}
Let $net(p,q)$ be the trained HNet. The training loss and test loss of $net(p,q)$, the original Hamiltonian $H(p,q)$, and the truncated inverse modified equation $MH1(p,q)$, $MH2(p,q)$ are given in Table~\ref{tab:errMSE}. The loss of $H(p, q)$ is much more larger than others, and the loss of modified Hamiltonian markedly decreases with the increasing of the truncation order. Fig.~\ref{fig:SEPD} presents three phase flows starting at $(0, 1)$ for $t=10$, and also show the conservation of the three Hamiltonians. The above results show that the network target is indeed the calculated Hamiltonian of the inverse modified equation rather than the original Hamiltonian.

\subsection{Symplectic HNets}

We call the HNet based on symplectic (non-symplectic) integrator as symplectic (non-symplectic) HNet. In this subsection, we will confirm that the symplectic HNets have better generalization ability and higher accuracy than the non-symplectic HNets in addressing predicting issues. In experiments, we
use a series of phase points $\{x_{i}\}_{i=1}^{n}$ with time step $h$ as the training data, i.e., $\mathcal{T}=\{(x_{i-1},x_{i})\}_{1}^{n}$ subject to $x_i=\phi_h(x_{i-1})$. Here symplectic HNets choose the implicit midpoint rule \cite[Chapter \uppercase\expandafter{\romannumeral2}.1]{hairer2006geometric}
\begin{equation*}
\bar{y} = y + h J^{-1} \nabla H(\frac{\bar{y}+y}{2}),
\end{equation*}
while non-symplectic HNets choose the implicit trapezoidal rule \cite[Chapter \uppercase\expandafter{\romannumeral2}.1]{hairer2006geometric}
\begin{equation*}
\bar{y} = y + \frac{h}{2}( J^{-1} \nabla H(\bar{y}) + J^{-1} \nabla H(y)).
\end{equation*}
Note that both of them are of order 2.
\subsubsection{Pendulum}
\begin{figure}[!htb]
    \centering
    \includegraphics[width=0.98\textwidth]{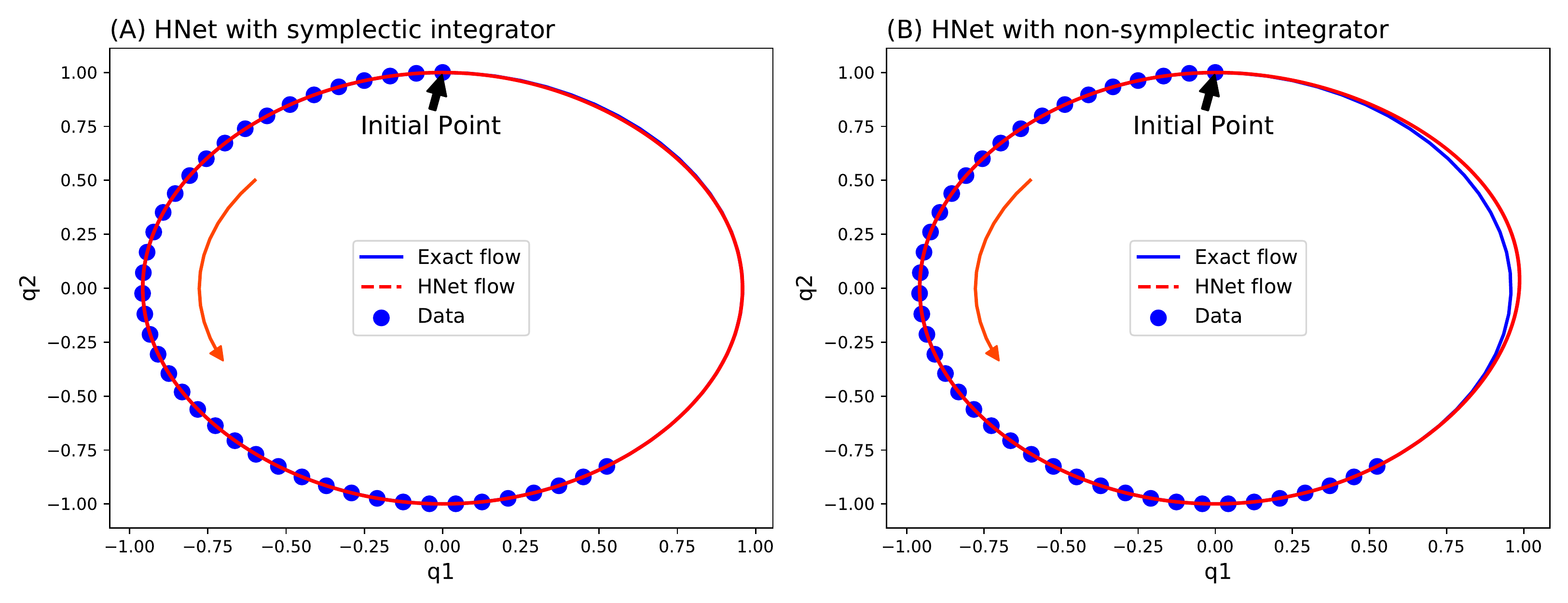}
    \caption{\textbf{Pendulum.} Comparison between the predicted flows of the symplectic HNet and the non-symplectic HNet. (\textbf{A}) shows the flow obtained by symplectic HNet, which discovers the unknown trajectory successfully. (\textbf{B}) shows the flow obtained by non-symplectic HNet, which deviates from the true trajectory.}
    \label{fig:PDqp}
\end{figure}
\begin{figure}[!htb]
    \centering
    \includegraphics[width=0.98\textwidth]{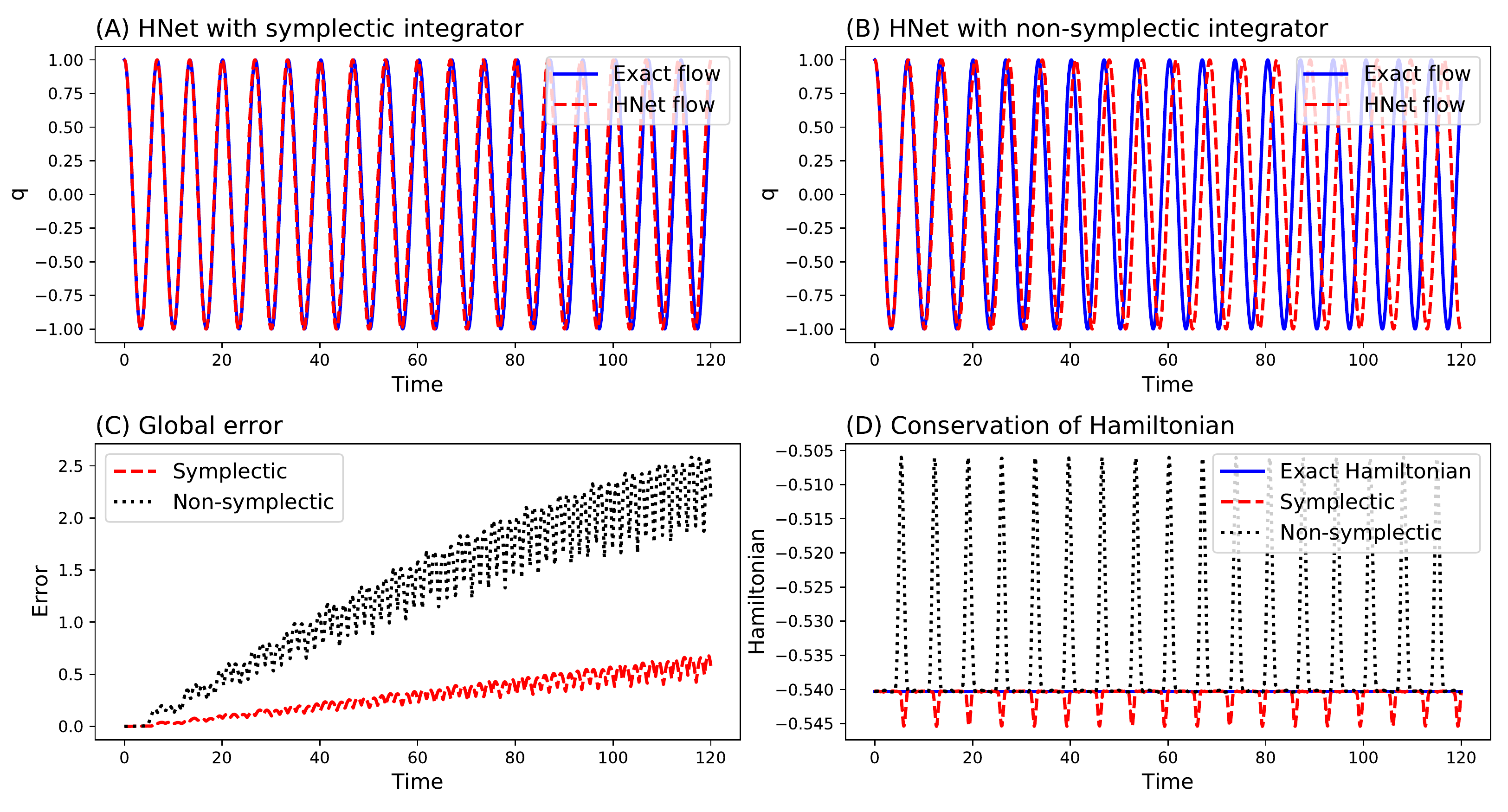}
    \caption{\textbf{Pendulum.} (\textbf{A}, \textbf{B}) Positions obtained by the symplectic HNet and the non-symplectic HNet. (\textbf{C}, \textbf{D}) Global error and conservation of Hamiltonian for HNets. Symplectic HNet gives comparatively accurate result.}
    \label{fig:PDqt}
\end{figure}
For pendulum, we obtain the flow starting from $x_0 = (0,1)$ with 40 points and time step $h=0.1$, as the training data, i.e., $\mathcal{T}=\{(x_{i-1},x_{i})\}_{1}^{40}$, $i=1,\cdots,40$. As shown in Fig.~\ref{fig:PDqp}, \ref{fig:PDqt}, the symplectic HNet reproduces the phase flow more accurately, which has lower global error and more accurate conservation of Hamiltonian.

\subsubsection{Kepler problem}
\begin{figure}[!htb]
    \centering
    \includegraphics[width=0.98\textwidth]{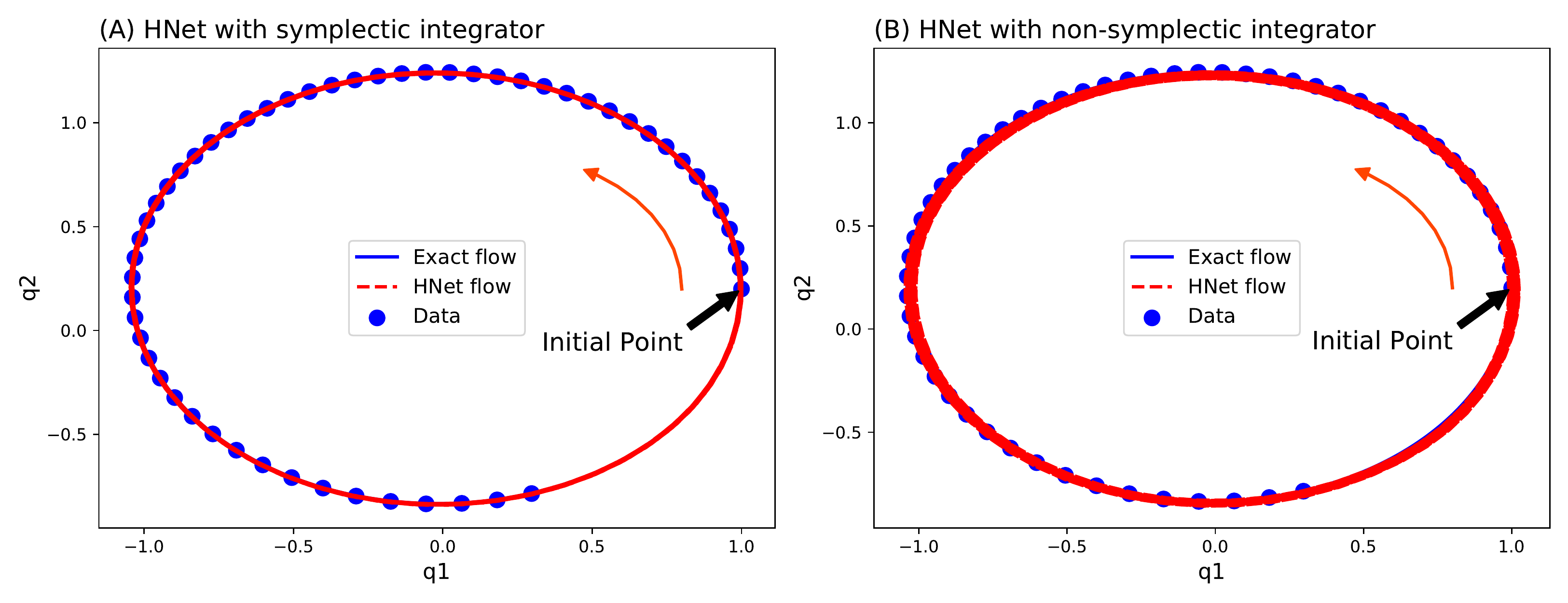}
    \caption{\textbf{Kepler problem.} Comparison between the predicted flows of the symplectic HNet and the non-symplectic HNet. (\textbf{A}) shows the flow obtained by symplectic HNet, which discovers the unknown trajectory successfully. (\textbf{B}) shows the flow obtained by non-symplectic HNet, which deviates from the true trajectory over time.}
    \label{fig:OSqp}
\end{figure}
\begin{figure}[!htb]
    \centering
    \includegraphics[width=0.98\textwidth]{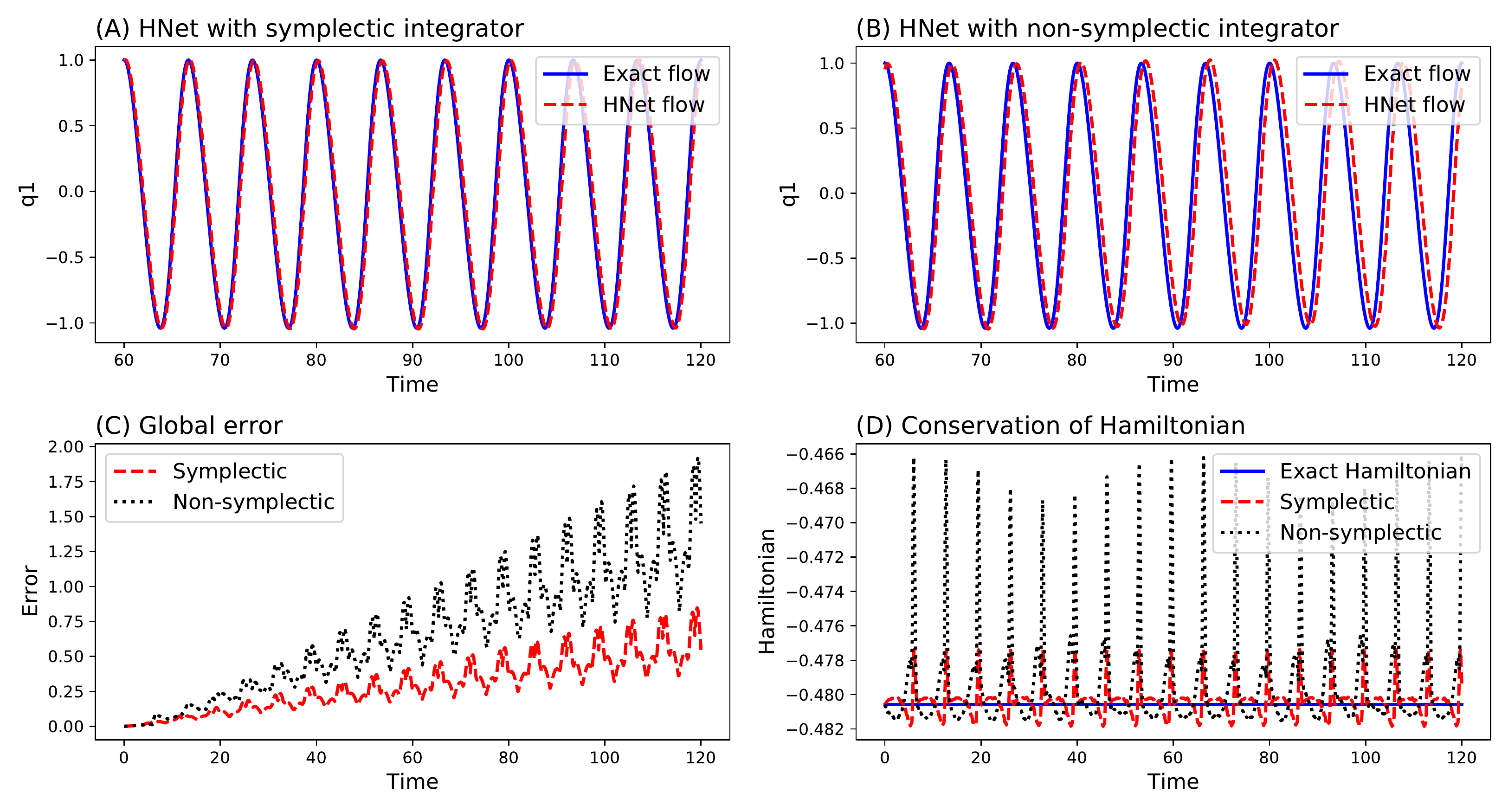}
    \caption{\textbf{Kepler problem. }(\textbf{A}, \textbf{B}) Positions obtained by the symplectic HNet and the non-symplectic HNet. Both HNets reproduce the phase portrait while symplectic HNets more accurately over time. (\textbf{C}, \textbf{D}) Global error and conservation of Hamiltonian for HNets. Symplectic HNet gives comparatively accurate result.}
    \label{fig:OSqt}
\end{figure}
Now we consider a four-dimensional system, the Kepler problem (mass $m = 1$, $M = 1$, gravitational constant = 1), which has the Hamiltonian
\begin{equation*}
H(\mathbf{p},\mathbf{q})=H(p_{1},p_{2},q_{1},q_{2})=\frac{1}{2}(p_{1}^{2}+p_{2}^{2})-\frac{1}{\sqrt{q_{1}^{2}+q_{2}^{2}}}.
\end{equation*}
We obtain the flow starting from $x_0 = (0,1,1,0.2)$ with 55 points and time step $h=0.1$, as the training data, i.e., $\mathcal{T}=\{(x_{i-1},x_{i})\}_{1}^{55}$, $ i=1,\cdots,55$. As shown in Fig.~\ref{fig:OSqp}, \ref{fig:OSqt}, the symplectic HNet reproduces the phase flow and captures the dynamic more accurately, which has lower global error and more accurate conservation of Hamiltonian.

\section{Conclusion}
This work explains the influences of different integrators as hyper-parameters on the HNets through error analysis. The target error is introduced to describe the gap between the network target and the true target, and the inverse modified equation is proposed to calculate the network target. The target error depends on the accuracy order of the integrator. Theoretical analysis shows that the HNets based on symplectic integrators possess network targets while non-symplectic integrators cannot guarantee the existence of the network targets. Numerical results have confirmed our theoretical analysis. HNets based on the symplectic integrators are learning the network targets rather than the Hamiltonian of the original system. In addressing predicting issues, symplectic HNets have better generalization ability and higher accuracy.

\bibliographystyle{abbrv}
\bibliography{main}

\end{document}